\documentclass[12pt, reqno]{amsart}
\usepackage{amsmath, amsthm, amscd, amsfonts, amssymb, graphicx, color}
\usepackage[bookmarksnumbered, colorlinks, plainpages]{hyperref}

\makeatletter
\@namedef{subjclassname@2010}{%
	\textup{2010} Mathematics Subject Classification}
\makeatother
\usepackage[mathscr]{euscript}
\usepackage{latexsym}
\usepackage{multicol,tabularx}
\usepackage{tikz}
\usepackage{graphics}
\usepackage{mathtools}

\usepackage{amsfonts}
\usepackage{amssymb}
\usepackage[english]{babel}
\usepackage{multido}
\usepackage[nomessages]{fp}
\usepackage[margin=2.5cm]{geometry}
\usepackage{enumitem}
\newtheorem{thm}{Theorem}[section]
\newtheorem{prop}{Proposition}[section]
\newtheorem{exam}{Example}[section]
\newtheorem{lem}{Lemma}[section]
\newtheorem{cor}{Corollary}[section]
\newtheorem{defn}{Definition}[section]

\numberwithin{equation}{section}
\begin{document}

	\setcounter{page}{1}
	\title[]{Normal Categories of Normed Algebra of Finite Rank Bounded Operators}
	\begin{abstract}
		In this article, we introduce the normal category $\mathscr L(\mathcal S)$ [$\mathscr R(\mathcal S)$] of principal left [right] ideals of the normed algebra $\mathcal S$ of all finite rank bounded operators on a Hilbert space $H$ and is shown that they are isomorphic,  using Hilbert space duality. We also described the semigroup of all normal cones in $\mathscr L(\mathcal S)$ which is isomorphic to the semigroup of all finite rank operators on $H$. Further, we construct bounded normal cones in $\mathscr L(\mathcal S)$ such that the set of all bounded normal cones in $\mathscr L(\mathcal S)$ is a normed algebra isomorphic to the normed algebra $\mathcal S$.
		
	\end{abstract}
	\author[P. G. Romeo, A. Anju ]{P. G. Romeo$^{1}$,     A. Anju$^{2}$}
	
	\address{$^{1}$ Department of Mathematics, Cochin University of Science And Technology, Kerala, India.}
	\noindent \email{romeo\_ parackal@yahoo.com
	}
	
	\address{$^{2}$ Department of Mathematics, Cochin University of Science And Technology, Kerala, India.}
	\noindent \email{anjuantony842@gmail.com
	}

	
	
	\date{ \newline \indent $^{2}$ Corresponding author}
	
	\baselineskip=21pt
	\vspace*{-1 cm}
	\maketitle
	
	\section{Introduction}
	\noindent In 1974, P. A. Grillet and K. S. S. Nambooripad \cite{g} introduced the cross-connections of regular partially ordered sets to study the structure of fundamental regular semigroups. Subsequently, K. S. S. Nambooripad \cite{s} 	extended this cross-connection to describe the structure of arbitrary regular semigroups using normal categories rather than regular partially ordered sets. 
	
	In this paper, we consider the regular semigroup $\mathcal S$ of all finite rank bounded operators on a Hilbert space $H$ as a normed algebra and construct bounded normal cones in the category of principal left ideals of 
	$\mathcal S$ in such a way that the cones forms a normed algebra isomorphic to the normed algebra $\mathcal S$. For this, consider the normal category $\mathscr L(\mathcal S) [\mathscr R(\mathcal S)]$ of principal left [right] ideals of the normed algebra $\mathcal S$ whose morphisms  are bounded linear maps and each hom-set of $\mathscr L(\mathcal S) [\mathscr R(\mathcal S)]$ is a normed space. 
	Note that, the normal category $\mathscr L(\mathcal S)$ is isomorphic to the normal category $\mathscr F(H)$ of finite dimensional subspaces of $H$ with linear maps as morphisms, and this isomorphism preserves the normed space structure of hom-sets.
	Similarly, the normal category $\mathscr R(\mathcal S)$ is isomorphic to the normal category $\mathscr F(H')$ of finite dimensional subspaces of the dual of $H$ with linear maps as morphisms. 
	Further, using Hilbert space duality it is shown that the normal categories $\mathscr L(\mathcal S)$ and $\mathscr R(\mathcal S)$ are isomorphic, and restriction of this isomorphism to hom-sets is a surjective conjugate-linear isometry. 
	Finally, it is shown that the regular semigroup of all normal cones in 
	$\mathscr L(\mathcal S)$ is isomorphic to the semigroup of all finite rank operators on 
	$H$ and we 
	described the bounded normal cone in $\mathscr L(\mathcal S)$ such that the set of all bounded normal cones in $\mathscr L(\mathcal S)$ produces a normed algebra isomorphic to the normed algebra $\mathcal S$. Similarly, the bounded normal cones in $\mathscr R(\mathcal S)$ is also a normed algebra which is conjugate-linear isomorphic to the normed algebra $\mathcal S$.
	
	\section{Preliminaries}
	\noindent The following we recall provide definitions and results regarding cross-connections needed in the sequel, for more details see \cite{kss}, \cite{sa}, \cite{s}. \\
	Let $\mathscr C$ to be a category. Then $\textit{\textbf v} \mathscr C$ denotes the object class of $\mathscr C$ and $\mathscr C$ itself denotes the morphism class of $\mathscr C$. For objects $a, b \in \textit{\textbf v} \mathscr C$, the collection of all morphisms from $a$ to $b$ is $\mathscr C(a, b)$, called a hom-set and for each $a \in \textit{\textbf v} \mathscr C,\, 1_a$ is the identity morphism in $\mathscr C(a, a)$. Let the identity functor on the category $\mathscr C$ be $1_{\mathscr C}$. A category $\mathscr C$ is termed as a small category if the class of all morphisms in $\mathscr C$ is a set.
	
	A category $\mathscr P$ is a preorder if for any $p, p' \in \textit{\textbf v} \mathscr P$, the hom-set $\mathscr P(p, p')$ contains at most one morphism. Let $\mathscr P$ be a preorder, then the relation $\subseteq$ on $\textit{\textbf v} \mathscr P$ defined by $p\subseteq p'$ if and only if $\mathscr P(p, p')\neq \phi$ is a quasiorder. A preorder $\mathscr P$ is said to be a strict preorder if $\subseteq$ on $\textit{\textbf v} \mathscr P$ is a partial order.
	\begin{defn}\emph{\cite{s}}
		A small category $\mathscr C$ is called a category with subobjects if $\mathscr C$ has a strict preorder $\mathscr P$ as a subcategory with $\textit{\textbf v} \mathscr P=\textit{\textbf v} \mathscr C$ such that every morphism in $\mathscr P$ is a monomorphism in $\mathscr C$ and if $f=hg$ for $f, g\in \mathscr P$, $h\in \mathscr C$, then $h \in \mathscr P$.
	\end{defn}
	\noindent In this case, the morphisms in $\mathscr P$ are called inclusions in $\mathscr C$, and a morphism $q$ in $\mathscr C$ is a retraction if $q$ is a right inverse of an inclusion in $\mathscr C$. A normal factorization of a morphism $f$ in a category $\mathscr C$ with subobjects is a factorization of the form $f=quj$ where $q$ is a retraction, $u$ is an isomorphism, and $j$ is an inclusion. If $f=quj$ is a normal factorization of $f$, then $f^\circ=qu$ is said to be the epimorphic component of $f$.
	
  Let $\mathscr C$ be a category with subobjects. A cone $\gamma$ in $\mathscr C$ with vertex $d\in \textit{\textbf v} \mathscr C$ is a map from $\textit{\textbf v} \mathscr C$ to $\mathscr C$ such that $\gamma(c)\in \mathscr C(c, d)$ $\forall c\in \textit{\textbf v} \mathscr C$, and if $j\in \mathscr C(c', c)$ is an inclusion in $\mathscr C$, then $j\gamma(c)=\gamma(c')$. For a cone $\gamma$ in $\mathscr C,\, c_\gamma$ denotes the vertex of $\gamma$ and for each $c\in \textit{\textbf v} \mathscr C$, the morphism $\gamma(c)$ is called the component of $\gamma$ at $c$. 
 A cone $\gamma$ with vertex $d$ is said to be a normal cone if at least one $c\in \textit{\textbf v} \mathscr C$ exists such that $\gamma(c)$ is an isomorphism from $c$ to $d$ \emph{\cite{a}}. 
	\begin{defn}\emph{(\cite{a},\cite{pa})}
		A normal category $\mathscr C$ is a category with subobjects such that every morphism in $\mathscr C$ has a normal factorization, every inclusion in $\mathscr C$ splits, and for each $d\in \textit{\textbf v} \mathscr C$, there is a normal cone $\gamma$ with vertex $d$ and $\gamma(d)=1_d$.
	\end{defn}

Two normal categories $\mathscr C$ and $\mathscr D$ are isomorphic if there exists inclusion preserving functors $F: \mathscr C \rightarrow \mathscr D$ and $G: \mathscr D \rightarrow \mathscr C$ such that $FG=1_{\mathscr C}$ and $GF=1_{\mathscr D}$ \noindent (\emph{\cite{s},\cite{t}}). 
If $\mathscr C$ is a normal category, then the set $T\mathscr C$ of all normal cones in $\mathscr C$ is a regular semigroup under the product defined by $(\gamma\cdot\sigma)(c)=\gamma(c)(\sigma(c_\gamma))^\circ$ for $\gamma, \sigma \in T\mathscr C$ and $c\in \textit{\textbf v} \mathscr C$, where $(\sigma(c_\gamma))^\circ$ is the epimorphic component of $\sigma(c_\gamma)$ (\emph{\cite{s},\cite{pa}}).

	\begin{prop}\label{p21}\cite{pa}
		If two normal categories $\mathscr C$ and $\mathscr D$ are isomorphic, then the semigroups $T\mathscr C$ and $T\mathscr D$ are isomorphic. 
	\end{prop}
	Let $S$ be a regular semigroup and $E(S)$ be the set of all idempotents in $S$.
		The category $\mathscr L(S)$ of principal left ideals of $S$ is a category whose objects are principal left ideals $Se$ of $S$ for $e\in E(S)$, and for objects $Se$ and $Sf$, the hom-set $\mathscr L(S)(Se, Sf)$ consists of partial right translations $\rho(e, u, f)$ for each $u\in eSf$, which maps $x\in Se$ to $xu\in Sf$ \cite{s}. 
		Let $S^{op}$ be the left-right dual of $S$. Then the category $\mathscr R(S)$ of principal right ideals of $S$ is $\mathscr L(S^{op})$ (\emph{\cite{kss},\cite{s}}).
	\begin{lem}\label{l21}\cite{s}
		Let $\mathscr L(S)$ be the category of principal left ideals of a regular semigroup $S$. Then,
		\begin{enumerate}
			\item $\mathscr L(S)$ is a category with subobjects whose inclusions are the usual set inclusions.
			\item For every $u, f \in E(S)$, the map $u\mapsto \rho(e, u, f)$ from $eSf$ to $\mathscr L(S)(Se, Sf)$ is a bijection.
			\item $\rho(e, u, f)=\rho(e', v, f')$ if and only if $e\mathscr L e'$, $f\mathscr L f'$, $u\in eSf$, $v\in e'Sf'$ and $v=e'u$.
			\item $\rho(e, u, f)$ is an isomorphism in $\mathscr L(S)$ if and only if $e \,\mathscr R\, u \,\mathscr L \,f$.
			\item If $Se\subseteq Sf$, then $j=\rho(e, e, f)$ is the inclusion from $Se$ to $Sf$, and $\rho$ is a retraction of $j$ if and only if $\rho=\rho(f, g, e)$ for some $g\in E(L_e)\cap \omega(f)$. 
			
		\end{enumerate}
		
	\end{lem}
		Let $a\in S$ and $f\in E(L_a)$. The map $\rho^a$ from $\textit{\textbf v} \mathscr L(S)$ to $\mathscr L(S)$ defined by $\rho^a(Se)=\rho(e, ea, f)$ for $Se\in \textit{\textbf v} \mathscr L(S)$ is a normal cone in $\mathscr L(S)$ with vertex $Sa$. The normal cones of the form $\rho^a$ in $\mathscr L(S)$ are called principal cones \cite{s}. 
	
	\begin{prop}\cite{s}
		Let $S$ be a regular semigroup. Then, the map $a\mapsto \rho ^a$ is a homomorphism from $S$ to $T\mathscr L(S)$. 
	\end{prop}
	
	\section{Ideal categories of semigroup of operators on a Hilbet space}
	
	\noindent It is well known that finite rank bounded operators on a Hilbert space $H$ over a field $\textbf K$, where $\textbf K$ is either $\mathbb R$ or $\mathbb C$ is a regular semigroup under the function composition $(T_1T_2)(x)=T_2(T_1(x))$ for $T_1,T_2\in\mathcal S$, and $x\in H$ and denote it by $S$. In the following, we discuss properties of $\mathscr L(\mathcal S)$ and $\mathscr R(\mathcal S)$ when $\mathcal S$ is regarded as a normed algebra.
	
	The object set $\textit{\textbf v} \mathscr L(\mathcal S)$ of the category $\mathscr L(\mathcal S)$ of principal left ideals of $\mathcal S$ is $$\textit{\textbf v} \mathscr L(\mathcal S)=\{\mathcal SP: P\in E(\mathcal S)\},$$ 
	and for $P_1, P_2 \in E(\mathcal S)$, the hom-set $hom\,(\mathcal S P_1, \mathcal S P_2)$ is 
	$$\mathscr L(\mathcal S)(\mathcal S P_1, \mathcal S P_2)=\{\rho(P_1, T, P_2): T\in P_1 \mathcal S P_2\},$$ where the morphism $\rho(P_1, T, P_2)$ maps $A\in \mathcal SP_1$ to $AT\in \mathcal SP_2$. 
	It is clear that each object $\mathcal SP$ is a left ideal of the normed algebra $\mathcal S$, and for $T\in P_1\mathcal SP_2$, the morphism $\rho(P_1, T, P_2)$ is the composition operator determined by $T$ and $\rho(P_1, T, P_2)$ is a bounded linear map.
	Dually  of principal right ideals of $\mathcal S$, with morphisms of the form  $\lambda(P_2, T, P_1)$ maps $A\in \mathcal P_1S$ to $TA\in \mathcal P_2S$  
	is the category  $\mathscr R(\mathcal S)$. 
	For $T\in \mathcal S$, $R(T)$ and $Z(T)$ denote the range space and the zero space of $T$, respectively. We often use the following lemma without indicating it. 
	\begin{lem} \cite{sv}
		If $T_1, T_2 \in \mathcal S$, then
		\begin{enumerate}[label=\emph{(\roman*)}]
			\item $\mathcal ST_1 \subseteq \mathcal ST_2$ if and only if $R(T_1)\subseteq R(T_2)$. 
			\item $T_1\mathcal S \subseteq T_2\mathcal S$ if and only if $Z(T_2)\subseteq Z(T_1)$.
		\end{enumerate}
	\end{lem}
	
	\noindent Hence, for principal left ideals $\mathcal ST_1$ and $\mathcal ST_2$, $\mathcal ST_1=\mathcal ST_2$ if and only if $R(T_1)=R(T_2)$, and for principal right ideals $T_1\mathcal S$ and $T_2\mathcal S$, $T_1\mathcal S=T_2\mathcal S$ if and only if $Z(T_1)=Z(T_2)$.

	It is known that $\mathcal S$ is an involution semigroup under the involution $T\mapsto T^*$, where $T^*$ is the adjoint of $T$ and an idempotent $P$ in $\mathcal S$ is a projection if and only if $Z(P)=R(P)^\perp$. Thus 
	the objects of $\mathscr L(\mathcal S)$ and $\mathscr R(\mathcal S)$ can uniquely be identified by the projections in $\mathcal S$.
	For every $\mathcal SP\in \textit{\textbf v} \mathscr L(\mathcal S)$, $R(P)$ is a finite-dimensional subspace of $H$, call it $M$, then $\mathcal SP=\mathcal SP_M$, where $P_M$ is the projection of $H$ onto $M$ and 
	$$\textit{\textbf v} \mathscr L(\mathcal S)=\{\mathcal S P_M: M \text{ is a finite-dimensional subspace of } H\}.$$
	Dually for the category $\mathscr R(\mathcal S)$ of principal right ideals of $\mathcal S$, if $P\in E(\mathcal S)$, then $P\mathcal S=P_M\mathcal S$, where $M=Z(P)^\perp$. For any finite-dimensional subspace $M$ of $H$, if $T\in \mathcal SP_M$, then $R(T)\subseteq M$. Conversely, if $T\in \mathcal S$ and $R(T)\subseteq M$, then $\mathcal ST\subseteq \mathcal SP_M$ implies $T\in \mathcal SP_M$. Hence, the operators in $\mathcal SP_M$ can be characterized as $\mathcal SP_M=\{T\in \mathcal S: R(T)\subseteq M\}.$ 
	\begin{align*}
	\text{But, } T\in P_M\mathcal S&\Longleftrightarrow T^*\in \mathcal SP_M\\
		&\Longleftrightarrow R(T^*)\subseteq M\\
		&\Longleftrightarrow M^\perp \subseteq Z(T).
	\end{align*}
	Thus, for finite-dimensional subspaces $M$ and $N$  of $H$, $T \in P_M\mathcal SP_N$ if and only if $T\in \mathcal S,\,  M^\perp \subseteq Z(T)$, and $R(T) \subseteq N$. Hence, morphisms in the hom-sets of the category $\mathscr L(\mathcal S)$ of principal left ideals of $\mathcal S$ can be identified as  $$\mathscr L(\mathcal S)(\mathcal SP_M, \mathcal SP_N)=\{\rho(P_M, T, P_N): T\in \mathcal S,\,  M^\perp \subseteq Z(T),\, R(T) \subseteq N\}.$$
	Summarizing the above results, we can state the following lemma.
	\begin{lem}\label{l32}
		Let $\mathscr L(\mathcal S)$ be the category of principal left ideals of $\mathcal S$. Then, the following holds.
		\begin{enumerate}
			\item $\textit{\textbf v} \mathscr L(\mathcal S)=\{\mathcal SP_M : M \text{ is a finite-dimensional subspace of } H\}$.
			\item If $M$ is a finite-dimensional subspace of $H$, then $\mathcal SP_M=\{T\in \mathcal S: R(T)\subseteq M\}.$ 
			\item If $M$ and $N$ are finite-dimensional subspaces of $H$, then the hom-set $$\mathscr L(\mathcal S)(\mathcal SP_M, \mathcal SP_N)=\{\rho(P_M, T, P_N): T\in \mathcal S,\,  M^\perp \subseteq Z(T),\, R(T) \subseteq N\}.$$
		\end{enumerate}
	\end{lem}	
\begin{exam}
	Let $H=\mathbb R^2$ and let $L$ be a line passing through the origin in $\mathbb R^2$. If $\theta$ is the angle between the line $L$ and the $X$-axis, then the projection of $H$ onto the line $L$ is given by $P_\theta=
	\begin{bmatrix}
		\cos^2\theta&\cos\theta \sin \theta\\
		\cos\theta \sin\theta&\sin^2\theta
	\end{bmatrix}.$ Thus, the objects of $\mathscr L(\mathcal S)$ are $\mathcal S P_H, \mathcal S P_{\{0\}},$ and $\mathcal S P_\theta$, where $0\leq\theta<\pi.$ 
\end{exam}
	\begin{exam}
		Let $H=\mathbb R^3, M=X$-axis and $N=YZ$-plane. Then,
		\begin{itemize}
			\item $\mathcal SP_M=\left\{\begin{bmatrix}a&0&0\\b&0&0\\c&0&0\end{bmatrix}: a, b, c\in \mathbb R\right\},$
			\item $\mathscr L(\mathcal S)(\mathcal SP_M, \mathcal SP_N)=\left\{\rho(P_M, T, P_N): T=\begin{bmatrix}
				0&a&b\\
				0&0&0\\
				0&0&0
			\end{bmatrix}; a, b\in \mathbb R\right\}.$
		\end{itemize}
	\end{exam}
	
	\noindent We now state the dual of the Lemma \ref{l32}.
	
	\begin{lem}
		Let $\mathscr R(\mathcal S)$ be the category of principal right ideals of $\mathcal S$. Then we have the following.
		\begin{enumerate}
			\item $\textit{\textbf v} \mathscr R(\mathcal S)=\{P_M\mathcal S : M \text{ is a finite-dimensional subspace of } H\}$.
			\item If $M$ is a finite-dimensional subspace of $H$, then $P_M\mathcal S=\{T\in \mathcal S: M^\perp \subseteq Z(T) \}$.
			\item If $M$ and $N$ are finite-dimensional subspaces of $H$, then the hom-set $$\mathscr R(\mathcal S)(P_M\mathcal S, P_N\mathcal S)=\{\lambda(P_M, T, P_N): T\in \mathcal S,\,  N^\perp \subseteq Z(T),\, R(T) \subseteq M\}.$$
		\end{enumerate}
	\end{lem}	
\noindent Next, we proceed to establish that the hom-sets of $\mathscr L(\mathcal S)$ are normed spaces. Let $M$ and $N$ be finite-dimensional subspaces  of $H$. If $T_1, T_2\in P_M\mathcal S P_N$, then $T_1+T_2,\, kT_1\in P_M\mathcal S P_N$, for all $k\in \textbf K$. Hence, 
	\begin{align*}
		\rho(P_M, T_1, P_N)+\rho(P_M, T_2, P_N)&=\rho(P_M, T_1+T_2, P_N),\text { and }\\
		k\rho(P_M, T_1, P_N)&=\rho(P_M, kT_1, P_N).
	\end{align*}
define addition and scalar multiplication on $\mathscr L(\mathcal S)(\mathcal SP_M, \mathcal SP_N)$. Clearly $\mathscr L(\mathcal S)(\mathcal SP_M, \mathcal SP_N)$ is a linear space and for $T\in P_M \mathcal S P_N$, $\rho=\rho(P_M, T, P_N)$ is a bounded linear map from the normed space $\mathcal SP_M$ to the normed space $\mathcal SP_N$ and the operator norm is a norm on $\mathscr L(\mathcal S)(\mathcal SP_M, \mathcal SP_N)$. Clearly,
	$$\Arrowvert \rho\Arrowvert  \leq \Arrowvert T \Arrowvert.$$
	$$\text { But, }\;\; \frac{\Arrowvert (P_M) \rho\Arrowvert }{\Arrowvert P_M \Arrowvert} =\Arrowvert P_MT \Arrowvert= \Arrowvert T \Arrowvert,$$
	therefore, $\mathscr L(\mathcal S)(\mathcal SP_M, \mathcal SP_N)$ is a normed space with respect to the norm $\Arrowvert \rho(P_M, T, P_N)\Arrowvert =\Arrowvert T\Arrowvert$. Thus, we have the following proposition.
	\begin{prop}
		Let $M$ and $N$ be finite-dimensional subspaces of $H$. Then the hom-set $\mathscr L(\mathcal S)(\mathcal SP_M, \mathcal SP_N)$ of $\mathscr L(\mathcal S)$ is a normed space.
	\end{prop}
	\noindent The dual of this proposition is:
	\begin{prop}
		If $M$ and $N$ are finite-dimensional subspaces of $H$, then the hom-set $\mathscr R(\mathcal S)(P_M\mathcal S, P_N\mathcal S)$ of $\mathscr R(\mathcal S)$ is a normed space under the functions
		\begin{align*}
			\lambda(P_M, T_1, P_N)+\lambda(P_M, T_2, P_N)&=\lambda(P_M, T_1+T_2, P_N),\\
			k\lambda(P_M, T, P_N)&=\lambda(P_M, kT, P_N), \text { and }\\
			\Arrowvert \lambda(P_M, T, P_N)\Arrowvert &=\Arrowvert T\Arrowvert,
		\end{align*}
		for  $T, T_1, T_2 \in P_N\mathcal SP_M$ and $k\in \textbf K.$
	\end{prop}
	
	It is known that for finite-dimensional subspaces $M$ and $N$ of $H$, $\mathcal SP_M\subseteq \mathcal SP_N$ if and only if $M \subseteq N$ and the norm of any inclusion is $1$. Next we show, for each inclusion in $\mathscr L(\mathcal S)$, there exists a unique retraction with norm $1$.
	\begin{prop}
		Let $M$ and $N$ be finite-dimensional subspaces of $H$ with $M\subseteq N$. Then $\rho(P_N, P_M, P_M)$ is the unique retraction to the inclusion $\rho(P_M, P_M, P_N)$ with norm $1$.
	\end{prop}
	\begin{proof}
		Let $P_1, P_2\in E(\mathcal S)$ with $R(P_1)=N$ and $R(P_2)=M$. If $\rho(P_1, P, P_2)$ is a retraction to $\rho(P_M, P_M, P_N)$ with norm $1$, then $P\in E(\mathcal S)$ and $\Arrowvert P \Arrowvert=1$, thus, $P$ is a projection. By Lemma \ref{l21}, $\rho(P_1, P, P_2)=\rho(P_N, P_NP, P_M) $.
		But, \begin{align*}
			\rho(P_M, P_M, P_M)&=\rho(P_M, P_M, P_N)\rho(P_1, P, P_2)=\rho(P_M, P_M, P_N)\rho(P_N, P_NP, P_M) \\
			&=\rho(P_M, P_MP_NP, P_M)=\rho(P_M, P_MP, P_N)
		\end{align*}
		hence, $P_M=P_MP$ by Lemma \ref{l21}. So, $P=P_M$ and $\rho(P_1, P, P_2)=\rho(P_N, P_NP_M, P_M)=\rho(P_N, P_M, P_M)$.
	\end{proof}
	
	For $T\in P_M\mathcal S P_N$, $\rho(P_M, T, P_N)$ is an isomorphism from $\mathcal SP_M$ to $\mathcal SP_N$ if and only if $P_M \,\mathscr R\, T \,\mathscr L\, P_N$ by Lemma \ref{l21}. But, $P_M \,\mathscr R\, T \,\mathscr L\, P_N$ if and only if $Z(T)=M^\perp$ and $R(T)=N$. Thus, the following proposition.
	\begin{prop}
		Let $T\in P_M\mathcal S P_N$. Then, $\rho(P_M, T, P_N)$ is an isomorphism from $\mathcal SP_M$ to $\mathcal SP_N$ if and only if $T|_M: M \rightarrow N$ is an isomorphism.
	\end{prop}
	
	Thus it is seen that normal factorization for each morphism $\rho(P_M, T, P_N)$ in $\mathscr L(\mathcal S)$ as $\rho(P_M, T, P_N)=\rho(P_M, P_U, P_U)\rho(P_U, T, P_V)\rho(P_V, P_V, P_N)$, where $U=Z(T)^\perp$ and $V=R(T)$. 
		
	\noindent As a left /right dual the proposition below follows
	\begin{prop}
		Let $M$ and $N$ be finite-dimensional subspaces of $H$. Then we have the following.
		\begin{enumerate}
			\item $P_M\mathcal S\subseteq P_N\mathcal S$ if and only if $M \subseteq N$.
			\item If $M\subseteq N$, then $\lambda(P_N, P_M, P_M)$ is the unique retraction of the inclusion $\lambda(P_M, P_M, P_N)$ with norm $1$.
			\item Let $T\in P_N\mathcal S P_M$. Then, $\lambda(P_M, T, P_N)$ is an isomorphism from $P_M\mathcal S$ to $P_N\mathcal S$ if and only if $T|_N$ is an isomorphism from $N$ onto $M$.
			\item For $T\in P_N\mathcal S P_M$, $\lambda(P_M, T, P_N)=\lambda(P_M, P_V, P_V)\lambda(P_V, T, P_U)\lambda(P_U, P_U, P_N)$, where $U=Z(T)^\perp$ and $V=R(T)$ is a normal factorization of $\lambda(P_M, T, P_N)$. 
		\end{enumerate}
	\end{prop}
	
	\section{Isomorphism between $\mathscr L(\mathcal S)$ and $\mathscr R(\mathcal S)$ }
	
	\noindent Let $\mathscr F(H)$ be the category whose vertices are finite-dimensional subspaces of a Hilbert space $H$. For finite-dimensional subspaces $M$ and $N$ of $H$, the hom-set $\mathscr F(H)(M, N)$ is the set of all linear maps from $M$ to $N$. Then $\mathscr F(H)$ is a category with subobjects in which inclusions are the usual subspace inclusion.
	For $M\subseteq N$, denote the inclusion map from $M$ to $N$ by $J_M^N$. Obviously, every inclusion in $\mathscr F(H)$ splits and for any linear map $T:M\longrightarrow N$, $T=PT_0J$ is a normal factorization of $T$, where $P$ is the projection of $M$ onto $Z(T)^\perp$, $T_0=T|_{Z(T)^\perp}$ is the isomorphism from $Z(T)^\perp$ onto $R(T)$ and $J$ is the inclusion map from $R(T)$ to $N$. Moreover, for each $M\in \textit{\textbf v} \mathscr F(H)$, the map $\gamma: \textit{\textbf v}\mathscr F(H) \longrightarrow \mathscr F(H)$ defined by $$\gamma(N)=(P_M)|_N: N \longrightarrow M, $$
	for $N\in \textit{\textbf v}\mathscr F(H)$ is a normal cone in $\mathscr F(H)$ with vertex $M$ and $\gamma(M)=1_M$,that is., $\mathscr F(H)$ is a normal category. Further, for each finite-dimensional subspaces $M$ and $N$ of $H$, the linear space $\mathscr F(H)(M, N)$ is a normed space with respect to the operator norm.  
	
	The following theorem says that the normal categories $\mathscr L(\mathcal S)$ and $\mathscr F(H)$ are isomorphic and the isomorphism preserves normed space structure of hom-sets. 
	
	\begin{thm}\label{t41}
		The normal category $\mathscr L(\mathcal S)$ is isomorphic to the normal category $\mathscr F(H)$, and for finite-dimensional subspaces $M$ and $N$ of $H$, this isomorphism is a normed space isomorphism from the hom-set $\mathscr L(\mathcal S)(\mathcal SP_M, \mathcal SP_N)$ of $\mathscr L(\mathcal S)$ onto the hom-set $\mathscr F(H)(M, N)$ of $\mathscr F(H)$.
	\end{thm}
	\begin{proof}
		To prove $\mathscr L(\mathcal S)$ is isomorphic to $\mathscr F(H)$ as normal categories, it will suffices to show that there exists 
		order-preserving functors $F: \mathscr L(\mathcal S) \rightarrow \mathscr F(H)$ and $G: \mathscr F(H) \rightarrow \mathscr L(\mathcal S)$ such that $FG=1_{\mathscr L(\mathcal S)}$ and $GF=1_{\mathscr F(H)}$.
		Define $F: \mathscr L(\mathcal S) \rightarrow \mathscr F(H)$ by,
		\begin{align*}
			\text{ for } \mathcal SP_M \in \textit{\textbf v}\mathscr L(\mathcal S), &\;F(\mathcal SP_M)=M,\\
			\text{ and for } \rho(P_M, T, P_N): \mathcal SP_M \rightarrow \mathcal SP_N, &\;F(\rho(P_M, T, P_N))=T|_M: M \rightarrow N.
		\end{align*}
		If $T_1\in P_M\mathcal S P_N$ and $T_2\in P_N\mathcal S P_U$, then \begin{align*}
			F(\rho(P_M, T_1, P_N)\rho(P_N, T_2, P_U))&=F(\rho(P_M, T_1T_2, P_U))=(T_1T_2)|_M =(T_1|_M)(T_2|_N)\\
			&=F(\rho(P_M, T_1, P_N))F(\rho(P_N, T_2, P_U)),
		\end{align*}
		
		\noindent and $F(1_{\mathcal S P_M}) $ $ =F(\rho(P_M, P_M, P_M))=P_M|_M=1_M$. Therefore, $F$ is a functor from $\mathscr L(\mathcal S)$ to $\mathscr F(H)$.
		
		\noindent Similarly, define $G$ from $\mathscr F(H)$ to $\mathscr L(\mathcal S)$ 
		by $G(M)=\mathcal SP_M$ where $M\in \textit{\textbf v}\mathscr F(H)$, and for 
		$T:M\rightarrow N, \;G(T)=\rho(P_M, P_MTP_N, P_N): \mathcal SP_M \rightarrow 
		\mathcal SP_N.$
		For $T_1: M\rightarrow N$ and $T_2: N\rightarrow U$,
		\begin{align*}
			G(T_1T_2)&=\rho(P_M, P_MT_1T_2P_U, P_U) =\rho(P_M, P_MT_1P_NT_2P_U, P_U) \\ &=\rho(P_M, P_MT_1P_N, P_N)\rho(P_N, P_NT_2P_U, P_U)=G(T_1)G(T_2),
		\end{align*}
		and $G(1_M)=\rho(P_M, P_M1_MP_M, P_M)=\rho(P_M, P_M, P_M)=1_{\mathcal S P_M}$, 
		ie., $G$ is a functor from $\mathscr F(H)$ to $\mathscr L(\mathcal S)$. 
		
	Since $$\mathcal SP_M \subseteq \mathcal SP_N  \Longleftrightarrow M\subseteq N,$$ 
	$F$ and $G$ are inclusion preserving functors with $FG=1_{\mathscr L(\mathcal S)}$ and $GF=1_{\mathscr F(H)}$.
		
		Finally, for finite-dimensional subspaces $M$ and $N$  of $H$, $F$ is a normed space isomorphism from  $\mathscr L(\mathcal S)(\mathcal SP_M, \mathcal SP_N)$ onto $\mathscr F(H)(M, N)$. For, let $T, T_1, T_2\in P_M\mathcal S P_N$ and $k\in \textbf K$. Then,
		\begin{align*}
			F(\rho(P_M, T_1, P_N)+\rho(P_M, T_2, P_N))&=F(\rho(P_M, T_1+T_2, P_N))=(T_1+T_2)|_M=(T_1|_M)+(T_2|_M)\\
			&=F(\rho(P_M, T_1, P_N))+F(\rho(P_M, T_2, P_N)), and
		\end{align*}
		\begin{align*}
			F(k\rho(P_M, T, P_N))&= F(\rho(P_M, kT, P_N))= (kT)|_M=k(T|_M)= k F(\rho(P_M, T, P_N)),\\
			\text{and }\Arrowvert F(\rho(P_M, T, P_N))\Arrowvert&= \Arrowvert T|_M \Arrowvert= \Arrowvert P_MT \Arrowvert=\Arrowvert T \Arrowvert=\Arrowvert \rho(P_M, T, P_N)\Arrowvert.
		\end{align*}
		Since $F$ is full and faithful, $F$ is a normed space isomorphism from  the hom-set $\mathscr L(\mathcal S)(\mathcal SP_M, \mathcal SP_N)$ of $\mathscr L(\mathcal S)$ onto the hom-set $\mathscr F(H)(M, N)$ of $\mathscr F(H)$.
	\end{proof}
	
	Let $H'$ denote the dual of the Hilbert space $H$. The category $\mathscr F(H')$ of all finite-dimensional subspaces of $H'$.
	 is a normal category. $M$ be a finite-dimensional subspace of $H$, for each $m\in M$, define the map $f_m:H\rightarrow \textbf K$ by
	 $f_m(x)=\langle x,m \rangle$ for $x\in H$. For $M'=\{f_m\in H': m\in M\}$, by Riesz representation theorem, $$\textit{\textbf v}\mathscr F(H')=\{M':M\in \textit{\textbf v}\mathscr F(H)\}.$$ Moreover, for finite-dimensional subspaces $M$ and $N$ of $H$, $\phi:M'\rightarrow N'$ is a linear map if and only if there exists a unique linear map $T:M\rightarrow N$ such that $\phi=(P_NT^*P_M)'|_{M'}$, where $T^*$ is the adjoint of $T$ and $(P_NT^*P_M)'$ is the transpose of $P_NT^*P_M$.
	 Then, $F:\mathscr F(H)\to \mathscr F(H')$ defined by 
	 \begin{align*}
	 	 F(M)&=M',\,\text({for}\,M\in \textit{\textbf v}\mathscr F(H)\\
	 	 F(T)&=(P_NT^*P_M)'|_{M'},\,\text{for every}\,\,T:M\to N
 		 \end{align*}
	 is a normal category isomorphism, for $T, T_1, T_2: M \rightarrow N$ and $k\in \textbf K$, 
	\begin{align*}
		F(T_1+T_2)&=(P_N(T_1+T_2)^*P_M)'|_{M'}=(P_NT_1^*P_M)'|_{M'}+(P_NT_2^*P_M)'|_{M'}=F(T_1)+F(T_2), \\
		F(kT)&= (P_N(kT)^*P_M)'|_{M'}= \bar k(P_NT^*P_M)'|_{M'}= \bar k F(T),\\
		\text{ and }  \Arrowvert F(T)\Arrowvert&= \Arrowvert (P_NT^*P_M)'|_{M'} \Arrowvert= \Arrowvert (P_NT^*P_M)' \Arrowvert = \Arrowvert P_NT^*P_M\Arrowvert = \Arrowvert T^* \Arrowvert= \Arrowvert T \Arrowvert. \end{align*}
	Therefore, $F$ is a conjugate linear isometry from the hom-set $\mathscr F(H)(M, N)$ of $\mathscr F(H)$ onto the hom-set $\mathscr F(H')(M', N')$ of $\mathscr F(H')$.
	Thus, we obtain the following theorem using the Hilbert space duality.
	
	\begin{thm}\label{t51}
		The normal categories $\mathscr F(H)$ and $\mathscr F(H')$ are isomorphic, and this isomorphism is a conjugate linear isometry from the hom-set $\mathscr F(H)(M, N)$ of $\mathscr F(H)$ onto the hom-set $\mathscr F(H')(M', N')$ of $\mathscr F(H')$ for any finite-dimensional subspaces $M$ and $N$ of $H$.
	\end{thm}

\noindent Analogous Theorem \ref{t41}, we have the following theorem.
	\begin{thm}\label{t42}
		The normal category $\mathscr R(\mathcal S)$ is isomorphic to the normal category $\mathscr F(H')$. For each finite-dimensional subspace $M$ and $N$ of $H$, this isomorphism is a normed space isomorphism of hom-set $\mathscr R(\mathcal S)(P_M\mathcal S, P_N\mathcal S)$ of $\mathscr R(\mathcal S)$ onto the hom-set $\mathscr F(H')(M', N')$ of $\mathscr F(H')$.
	\end{thm}
	\begin{proof}
		Define $F: \mathscr R(\mathcal S) \rightarrow \mathscr F(H')$ by,
		\begin{align*}
			\;F(P_M\mathcal S)=M\, & \text{ for } P_M\mathcal S \in 
			\textit{\textbf v}\mathscr R(\mathcal S)\\
			\text{ and for } \lambda(P_M, T, P_N): P_M\mathcal S \rightarrow P_N\mathcal S,& \;F(\lambda(P_M, T, P_N))=T'|_{M'}:M'\rightarrow N'.
		\end{align*}
		As in the proof of Theorem \ref{t41}, it is seen that $F$ is a normal category isomorphism from $\mathscr R(\mathcal S)$ to $\mathscr F(H')$, and also a normed space isomorphism from the hom-set $\mathscr R(\mathcal S)(P_M\mathcal S, P_N\mathcal S)$ of $\mathscr R(\mathcal S)$ onto the hom-set $\mathscr F(H')(M', N')$ of $\mathscr F(H')$ of finite-dimensional subspaces $M$ and $N$.
	\end{proof}
	
	\noindent From the above discussions we can conclude that the functor $F$ from $\mathscr L(\mathcal S)$ to $\mathscr R(\mathcal S)$ given by,
	\begin{align*}
		F(\mathcal SP_M )= P_M \mathcal S,\,&\,	\text{for\;} \mathcal SP_M \in \textit{\textbf v}\mathscr L(\mathcal S),\\
		\text{and for\;}\rho(P_M, T, P_N) : \mathcal SP_M \rightarrow \mathcal SP_N, &\;   F(\rho(P_M, T, P_N)) = \lambda(P_M, T^*, P_N)\end{align*}
	is a normal category isomorphism from $\mathscr L(\mathcal S)$ to $\mathscr R(\mathcal S)$ and for finite-dimensional subspaces $M$ and $N$ of $H$, $F$ is a  conjugate-linear isometry from the hom-set $\mathscr L(\mathcal S)(\mathcal SP_M, \mathcal SP_N)$ of $\mathscr L(\mathcal S)$ onto the hom-set $\mathscr R(\mathcal S)(P_M\mathcal S, P_N\mathcal S)$ of $\mathscr R(\mathcal S)$. Thus we have the following:
	\begin{thm}\label{t44}
		The normal category $\mathscr L(\mathcal S)$ is isomorphic to the normal category $\mathscr R(\mathcal S)$. Moreover, this isomorphism is a conjugate-linear isometry from the hom-set $\mathscr L(\mathcal S)(\mathcal SP_M, \mathcal SP_N)$ of $\mathscr L(\mathcal S)$ onto the hom-set $\mathscr R(\mathcal S)(P_M\mathcal S, P_N\mathcal S)$ of $\mathscr R(\mathcal S)$ for finite-dimensional subspaces $M$ and $N$ of $H$.
	\end{thm}
	
	\section{Bounded normal cones}
	\noindent It is well known that isomorphic normal categories have isomorphic semigroups of normal cones (see Proposition \ref{p21}). Since $\mathscr F(H)$, $\mathscr F(H')$, $\mathscr L(\mathcal S)$ and $\mathscr R(\mathcal S)$ are isomorphic as normal categories, it will sufficies to find the semigroup $T\mathscr F(H)$ of all normal cones in $\mathscr F(H)$.
	\begin{thm}\label{t61}
		The semigroup $T\mathscr F(H)$ of all normal cones in $\mathscr F(H)$ is isomorphic to the semigroup of all finite rank operators on the Hilbert space $H$. 
	\end{thm}
	\begin{proof}
		Let $\mathcal S'$ be the semigroup of all finite rank operators on $H$. For $T\in \mathcal S'$, define the map $\gamma^T$ from $\textit{\textbf v} \mathscr F(H)$ to $\mathscr F(H)$ by, $$\gamma^T(M)=T|_M: M\rightarrow R(T),$$ for $M\in \textit{\textbf v} \mathscr F(H)$. If $M_1 \subseteq M_2$, then $J_{M_1}^{M_2}\,\gamma^T(M_2)=J_{M_1}^{M_2}\,(T|_{M_2})=T|_{M_1} = \gamma^T(M_1),$ therefore, $\gamma^T$ is a cone in $\mathscr F(H)$. 		
		For any complement $M$ of $Z(T)$, $\,T|_M$ is an isomorphism from $M$ onto $R(T)$, hence, $\gamma^T$ is a normal cone with vertex $R(T)$. Thus, for any finite rank operator $T$on $H$, there is a normal cone $\gamma^T$ in $\mathscr F(H)$ with vertex $R(T)$.
		Define $\phi:\mathcal S'\rightarrow T\mathscr F(H)$ by $$\phi(T)=\gamma^T.$$ Let $T_1, T_2 \in \mathcal S'$ and $N=R(T_1)$. Then for every $M\in \textit{\textbf v} \mathscr F(H)$,
		\begin{align*}
			\phi(T_1T_2)(M)&=\gamma^{T_1T_2}(M)=(T_1T_2)|_M=(T_1|_M)(T_2|_{N})^\circ\\
			&=\gamma^{T_1}(M)(\gamma^{T_2}(N))^\circ=\gamma^{T_1} \cdot\gamma^{T_2}(M)\\
			&=(\phi(T_1)\cdot\phi(T_2))(M),
		\end{align*}
		that is, $\phi(T_1T_2)= \phi(T_1)\cdot\phi(T_2)$ and so $\phi$ is a homomorphism from $\mathcal S'$ to $T\mathscr F(H)$.
		
		Next, we have to show that for any normal cone $\gamma$ in $\mathscr F(H)$ with vertex $N$, there is a unique finite rank operator $T$ on $H$ such that $R(T)=N$ and $\gamma^T=\gamma$. Let $\gamma$ be a normal cone in $\mathscr F(H)$ with vertex $N$ and $\mathscr B$ be a basis of $H$. Define the map $T$ on $H$ by for $b\in \mathscr B$, 
		$$T(b)=\gamma(\langle b \rangle)(b),$$ 
		where $\langle b \rangle= \text{Span}\{b\}$, and extend $T$ linearly to all of $H$. That is, if $ T_b=\gamma(\langle b \rangle)$ for $b\in \mathscr B$, then $T(b)=T_b(b)$. To prove that $\gamma(M)=T|_M$ for any finite-dimensional subspace $M$ of $H$, it is enough to prove that $\gamma(\langle x \rangle)=T|_{\langle x \rangle}$ for any $x\in H$. Let  $x=c_1b_1+c_2b_2+\cdots +c_nb_n$ for $b_i\in \mathscr B$ and $c_i\in \textbf K$, $i=1, 2, \cdots, n$. Let $U=\text{Span}\{b_1, b_2, \cdots, b_n\}$, and $T_{b_i}=\gamma(\langle b_i \rangle)$ for $i=1, 2, \cdots, n$, then for $i=1, 2, \cdots, n$, since $\langle b_i \rangle\subseteq U$, $T(b_i)=T_{b_i}(b_i)=\gamma(\langle b_i \rangle) (b_i)=\gamma(U)(b_i)$, therefore, $\gamma(U)=T|_U$, and $\gamma(\langle x\rangle)=J_{\langle x \rangle}^U \, \gamma(U)=J_{\langle x \rangle}^U \, T|_U= T|_{\langle x \rangle}$ for any $x\in H$. 
		Since, $\gamma$ is a normal cone, there is a finite-dimensional subspace $M$ of $H$ such that $\gamma(M)=T|_M$ is an isomorphism from $M$ to $N$. Thus, $R(T)=N$ and $\gamma^T=\gamma$. 
		
		Finally to prove the uniqueness of $T$, suppose that $\gamma=\gamma^A$ for some $A\in \mathcal S'$. Then for each $b\in \mathscr B, (T|_{\langle b \rangle})(b)= \gamma^T(\langle b \rangle)(b)=\gamma^A(\langle b \rangle)(b)= (A|_{\langle b \rangle})(b)$. Thus, $T(b)=A(b)$ for each $b\in \mathscr B$. Therefore, $T=A$.
	Hence, $\phi$ is an isomorphism from the semigroup $\mathcal S'$ of all finite rank operators on $H$ to the semigroup $T\mathscr F(H)$ of all normal cones in $\mathscr F(H)$.
	\end{proof}
	Since the semigroup $T\mathscr F(H)$ of all normal cones in $\mathscr F(H)$ is a regular semigroup, we have the following.
	\begin{cor}
		The semigroup of all finite rank operators on $H$ is a regular semigroup.
	\end{cor}
	\begin{exam}
		Let $H=l^2$ and for $n=1, 2, 3, \ldots$, $e_n=(0, \ldots, 0, 1, 0, \ldots),$ where $1$ occurs only in the $nth$ entry. Also let $\mathscr B$ be a basis of $l^2$ containing $e_1, e_2, e_3, \ldots$. Define $T: l^2\rightarrow l^2$ by, 
		\begin{equation*}
			T(x)=
			\begin{cases}
				ne_1 & \text{ if } x=e_n\\
				0 & \text{ if } x\in \mathscr B, x\neq e_n \text{ for any } n.
			\end{cases}
		\end{equation*}
		By extending $T$ linearly to $l^2$, $T$ is a finite rank operator on $l^2$, which is not bounded. Then one can define a normal cone $\gamma^T$ in $\mathscr F(H)$ with vertex $\langle e_1\rangle=Span\{e_1\}$ by for $M\in \textit{\textbf v} \mathscr F(H)$, $$\gamma^T(M)=T|_M:M\rightarrow \langle e_1\rangle.$$
		
		Similarly it is easy to see that one can define normal cone $\rho^T$ in $\mathscr L(\mathcal S)$ with vertex $\mathcal SP_{\langle e_1\rangle}$ by, for $\mathcal SP_M \in \textit{\textbf v} \mathscr L(\mathcal S)$, $$\rho ^T(\mathcal SP_M)=\rho(P_M, P_MT, P_{\langle e_1\rangle}),$$
		and normal cone $\lambda^T$ in $\mathscr R(\mathcal S)$ with vertex $P_{\langle e_1\rangle}\mathcal S$ by,  for $P_M\mathcal S \in \textit{\textbf v} \mathscr R(\mathcal S)$, $$\lambda ^T(P_M\mathcal S)=\lambda(P_M, (P_MT)^*, P_{\langle e_1\rangle}).$$
	\end{exam}
	Next we proceed to descibe bounded normal cone in $\mathscr F(H)$ such that the set of all bounded normal cones in $\mathscr F(H)$ forms a normed algebra isomorphic to the normed algebra $\mathcal S$.
	\begin{defn}
		A normal cone $\gamma$ in $\mathscr F(H)$ with vertex $N$ is said to be a bounded normal cone if there exists some $\alpha_\gamma \in \mathbb R$ such that $ \Arrowvert  \gamma(M)\Arrowvert \leq \alpha_\gamma $, for all $M \in \textit{\textbf v} \mathscr F(H)$.
	\end{defn}
	Analogous to this, we can define bounded normal cones in $\mathscr L(\mathcal S)$ and $\mathscr R(\mathcal S)$.
	The following theorem revels that the set $B\mathscr L(\mathcal S)$ of all bounded normal cones in $\mathscr L(\mathcal S)$ is a normed algebra isomorphic to the normed algebra $\mathcal S$.
	\begin{thm}
		The set $B\mathscr L(\mathcal S)$ of all bounded normal cones in $\mathscr L(\mathcal S)$ is a normed algebra, which is isomorphic to the normed algebra $\mathcal S$.
	\end{thm}
	\begin{proof}
		For $T\in \mathcal S$ the principal cone $\rho ^T$ is a normal cone in $\mathscr L(\mathcal S)$ with vertex $\mathcal S P_N$, where $N=R(T)$ and for any 
		$\mathcal SP_M \in \textit{\textbf v} \mathscr L(\mathcal S)$, 
		$$\Arrowvert \rho ^T(\mathcal SP_M)\Arrowvert=\Arrowvert \rho(P_M, P_MT, P_N)\Arrowvert= \Arrowvert P_M T\Arrowvert \leq \Arrowvert T \Arrowvert ,$$ therefore, $\rho ^T$ is a bounded normal cone in $\mathscr L(\mathcal S)$. Define the function $\phi: \mathcal S\rightarrow B\mathscr L(\mathcal S)$ by, 
		$$\phi(T)=\rho^T,$$
		clearly  $\phi$ is one-one. To prove that $\phi$ is onto, suppose that $\gamma$ is a bounded normal cone in $\mathscr L(\mathcal S)$ with vertex $\mathcal S P_N$, by Theorem \ref{t61}, there exists a finite rank operator $T$ on $H$ such that $\gamma= \rho ^T$ and $N=R(T)$, ie., for any $\mathcal SP_M \in \textit{\textbf v} \mathscr L(\mathcal S)$, $\gamma(\mathcal SP_M)=\rho ^T(\mathcal SP_M)=\rho(P_M, P_MT, P_N)$. 
		
		\noindent Since $\rho ^T$ is a bounded normal cone, there exists some $\alpha \in \mathbb R$ such that $ \Arrowvert  \rho ^T(\mathcal SP_M)\Arrowvert \leq \alpha $, for all $\mathcal SP_M \in \textit{\textbf v} \mathscr L(\mathcal S)$, so for any $\mathcal SP_M \in \textit{\textbf v} \mathscr L(\mathcal S)$, $\Arrowvert  P_MT\Arrowvert = \Arrowvert \rho(P_M, P_MT, P_N)\Arrowvert= \Arrowvert \rho ^T(\mathcal SP_M)\Arrowvert \leq \alpha $. Hence, $\Arrowvert P_{\langle x \rangle}T\Arrowvert \leq \alpha$ for any $x\in H$. Thus, for any $x\in H$, 
		$$\Arrowvert T(x) \Arrowvert= \Arrowvert (P_{\langle x \rangle}T)(x)\Arrowvert \leq \alpha \Arrowvert x \Arrowvert, $$ 
		therefore, $T$ is a finite rank bounded operator on $H$ and $\gamma=\phi(T)$, and 
		$B\mathscr L(\mathcal S)$ is a normed space. Further for $\rho ^{T_1}, \rho ^{T_2}$ in $B\mathscr L(\mathcal S)$,
		$$\rho^{T_1}+\rho^{T_2}=\rho^{T_1+T_2},$$
		$$k\rho^{T_1}=\rho^{kT_1},$$
		$$ \text{ and } \Arrowvert \rho^{T_1}\Arrowvert= \Arrowvert T_1\Arrowvert.$$
		But,
		$$\phi(T_1T_2)=\rho^{T_1T_2}=\rho^{T_1}\cdot\rho^{T_2}=\phi(T_1)\cdot\phi(T_2),$$
		$$\phi(T_1+T_2)=\rho^{T_1+T_2}=\rho^{T_1}+\rho^{T_2}=\phi(T_1)+\phi(T_2),$$
		$$\phi(kT)=\rho^{kT}=k\rho^{T}=k\phi(T),$$
		$$ \text{ and }\Arrowvert \phi(T)\Arrowvert= \Arrowvert \rho^T\Arrowvert=\Arrowvert T\Arrowvert.$$
		Thus, the set $B\mathscr L(\mathcal S)$ of all bounded normal cones in $\mathscr L(\mathcal S)$ is a normed algebra that is isomorphic to the normed algebra $\mathcal S$.
	\end{proof}

	Consier the normal cone $\gamma'$ in $\mathscr R(\mathcal S)$ with vertex $P_N\mathcal S$ corresponding to the normal cone $\gamma$ in $\mathscr L(\mathcal S)$ with vertex $\mathcal S P_N$ given by for $P_M\mathcal S \in \textit{\textbf v} \mathscr R(\mathcal S)$,
	$$\gamma'(P_M\mathcal S)=F(\gamma(F^{-1}(P_M\mathcal S))),$$ 
	where $F$ is the isomorphism from $\mathscr L(\mathcal S)$ to $\mathscr R(\mathcal S)$ in Theorem \ref{t44}. 
	Since $F$ is an isometry on each hom-set, there is a bijection between the set $B\mathscr R(\mathcal S)$ of all bounded normal cones in $\mathscr R(\mathcal S)$, and the set $B\mathscr L(\mathcal S)$ of all bounded normal cones in $\mathscr L(\mathcal S)$.
	Furthermore, for $T\in \mathcal S$, the normal cone 
	corresponding to the normal cone $\rho^T$ in $\mathscr L(\mathcal S)$ is $\lambda^{T^*}$.
	For $P_M\mathcal S \in \textit{\textbf v} \mathscr R(\mathcal S)$, 
	$$\gamma'(P_M\mathcal S)=F(\rho^T(\mathcal SP_M))=F(\rho(P_M, P_MT, P_N))=\lambda(P_M, (P_MT)^*, P_N)=\lambda^{T^*}(P_M\mathcal S),$$
	hence, for $\lambda ^{T_1}, \lambda ^{T_2}$ in $B\mathscr R(\mathcal S)$. Also 
	$$\lambda^{T_1}+\lambda^{T_2}=\lambda^{T_1+T_2},$$
	$$k\lambda^{T_1}=\lambda^{kT_1},$$
	$$ \text{ and } \Arrowvert \lambda^{T_1}\Arrowvert= \Arrowvert T_1\Arrowvert.$$
	Thus we have following theorem. 
	\begin{thm}
		The set $B\mathscr R(\mathcal S)$ of all bounded normal cones in $\mathscr R(\mathcal S)$ is a normed algebra, and this normed algebra $B\mathscr R(\mathcal S)$ is conjugate linear isomorphic to the normed algebra $\mathcal S$.
	\end{thm}
	\begin{proof}
		Define the map $\phi: \mathcal S \rightarrow B\mathscr R(\mathcal S)$ by,
		$\phi(T)=\lambda^{T^*}$. Then,
		$$\phi(T_1T_2)=\lambda^{(T_1T_2)^*}=\lambda^{T_2^*T_1^*}=\lambda^{T_1^*}\cdot\lambda^{T_2^*}=\phi(T_1)\cdot\phi(T_2),$$
		$$\phi(T_1+T_2)=\lambda^{(T_1+T_2)^*}=\lambda^{(T_1)^*+(T_2)^*}=\lambda^{T_1^*}+\lambda^{T_2^*}=\phi(T_1)+\phi(T_2),$$
		$$\phi(kT)=\lambda^{(kT)^*}=\lambda^{\bar kT^*}=\bar k\lambda^{T^*}=\bar k\phi(T),$$
		$$\text {and }\Arrowvert \phi(T)\Arrowvert= \Arrowvert \lambda^{T^*}\Arrowvert=\Arrowvert T^*\Arrowvert=\Arrowvert T\Arrowvert,$$
		therefore, $B\mathscr R(\mathcal S)$ is a normed algebra, and it is conjugate linear isomorphic to the normed algebra $\mathcal S$.
	\end{proof}
	
\end{document}